\documentclass[12pt]{amsart}
\usepackage{amssymb,epsfig,amsmath,latexsym,amsthm}
\usepackage{graphicx}
\usepackage{enumitem}
\usepackage{comment}
\theoremstyle{plain}
\newtheorem{theorem}{Theorem}[section]
\newtheorem{lemma}[theorem]{Lemma}

\newtheorem{proposition}[theorem]{Proposition}
\newtheorem{corollary}[theorem]{Corollary}

\theoremstyle{definition}
\newtheorem{definition}[theorem]{Definition}
\newtheorem{remark}[theorem]{Remark}

\usepackage{epsfig,color}

\makeatother

\theoremstyle{definition}

\def\fnum{equation}

\DeclareMathOperator{\sech}{sech}

\numberwithin{equation}{section}

\usepackage[hidelinks]{hyperref}
\usepackage[margin=1.4in]{geometry}
\begin{document}

	\title[Cheeger constant of convex co-compact hyperbolic 3-manifolds]{Cheeger constant of convex co-compact hyperbolic 3-manifolds}

	\author{Franco Vargas Pallete and Celso Viana}
	
	\address{
		Institut des Hautes \'Etudes Scientifiques, Bures-sur-Yvette, 91440 France
	}
	\email{vargaspallete@ihes.fr}
	
	\address{
		Universidade Federal de Minas Gerais, Belo Horizonte, 30123-970, Brazil
	} 
	\email{celso@mat.ufmg.br}
	
	\thanks{FVP was partially supported by NSF grant DMS-2001997. CV was partially supported by CNPq Grant 405468/2021-0}
	
	\begin{abstract}
		We study the Cheeger isoperimetric constant as a functional in the space of convex co-compact hyperbolic 3-manifolds which are  either  quasi-Fuchsian or acylindrical. We prove that the global  maximum is attained uniquely  at the Fuchsian locus.
	\end{abstract}

	\maketitle

	\section{Introduction}
	\noindent Given a hyperbolic manifold $\mathbb{H}^3/\Gamma$, we define its limit set $\Lambda(\Gamma)$ as the set of accumulation points in $\partial_\infty\mathbb{H}^3 = \mathbb{S}^2$ of an orbit $\Gamma_x$. The set $\Lambda(\Gamma)$ does not depend on the basepoint $x\in\mathbb{H}^3$. We define the \textit{convex core} as the quotient by $\Gamma$ of the convex hull of $\Lambda(\Gamma)$. A  non-compact hyperbolic 3-manifold is called \textit{convex co-compact} if its convex core is compact. Important examples include Schottky,  quasi-Fuchsian and those which are acylindrical. We will detail the last two examples as they are relevant to our results.
	
	We say $\mathbb{H}^3/\Gamma$ is  \textit{quasi-Fuchsian} if its limit set $\varLambda(\Gamma)$ is a closed Jordan curve in $\partial_{\infty}\mathbb{H}^3=\mathbb{S}^2$. Equivalently, $\Gamma$ is a quasi-conformal deformation of a Fuchsian group $\Gamma_0$, i.e.,  a subgroup of $PSL(2,\mathbb{R})$ representing the fundamental group of a closed hyperbolic surface. 
	A famous result states that the Hausdorff dimension of $\varLambda(\Gamma)$ is greater than one unless  $\Gamma$ is  Fuchsian, see \cite{Bowen}. Geometrically, a Fuchsian 3-manifold  $\mathbb{H}^3/\Gamma_0$  is simply  $\Sigma\times \mathbb{R}$  endowed with   the   warped product metric:
	\begin{equation}\label{fuchsian_metric}
		g_F= dr^2 + \cosh^2(r)\,g_{\Sigma},
	\end{equation}
	where $\Sigma$ is the  hyperbolic surface $\mathbb{H}^2/\Gamma_0$. A convex co-compact hyperbolic  $(M^3,\partial M)$ is  \textit{acylindrical} if $\partial M$ is incompressible and if
	every map of the cylinder $C= \mathbb{S}^1\times I$, $f: (C,\partial C)\rightarrow (M,\partial M)$ which takes the components of $\partial C$ to non-trivial homotopy classes in $\partial M$, is homotopic into $\partial M$.  A nice feature that unite acylindrical and quasi-Fuchsian  is the existence of  model geometries when the  topology is fixed. These  model geometries are characterized by having convex core with totally geodesic boundary components, we refer to constructions on these models by using the sub-index $TG$.
	
	In this note we investigate the Cheeger isoperimetric constant in the space of  quasi-Fuchsian  and acylindrical 3-manifolds.
	Recall that the Cheeger  constant of a complete non-compact manifold $M$  is defined as
	\[
	h(M)= \inf \bigg\{ \frac{A(\partial \Omega)}{V(\Omega)}\,\,:\,\, \Omega \subset M \bigg\}
	\]
	where $\Omega$ has compact closure and smooth boundary in $M$. The interest  in this number comes from the classical Cheeger inequality \cite{C,Y}:
	\begin{equation}\label{cheeger_ineq}
		\frac{h^2(M)}{4}\leq \lambda_0(M),
	\end{equation}
	where $\lambda_0$ is the bottom spectrum of the Laplacian of $M$ definded by:
	\[
	\lambda_0(M):=\inf_{f\in C^{\infty}_0(M)} \frac{\int_M |\nabla f|^2\, dM}{\int_M f^2\,dM}.
	\] Computing lower bounds for $h(M)$ is an interesting but difficult problem in geometry;   see  \cite{Schoen,Y} concerning simply connected spaces with upper sectional curvature bounds. A general principle, Buser \cite{Buser}, states that in   manifolds with $Ric_M\geq -(n-1)\kappa$, the numbers $\lambda_0(M)$ and $h(M)$  are essentially equivalent in the sense that $\lambda_0(M)\leq c_n\kappa\, h(M)$ for some dimensional constant $c_n$.
	The literature also shows that the quantities $h(M)$, $\lambda_0$,  the Hausdorff dimension of $\varLambda(\Gamma)$, and other invariants such as the critical exponent  $\delta(\Gamma)$ are all interrelated when the context involves Kleinian  groups \cite{S}.
	
	The study of $\lambda_0$ in the  class  of conformally compact Einstein manifolds  was done in \cite{W}. For convex co-compact hyperbolic 3-manifolds, the work   \cite{W}  shows that  if $ H_2(M^{3},\mathbb{Z})\neq 0$, then $\lambda_0< 1$ unless $M$ is Fuchsian ($\lambda_0=1$);  this  particular class  also follows from results in \cite{Bowen,S}, see Remark 1.2 below. For the general setting of complete manifolds $M^n$ satisfying $Ric_M\geq -(n-1)$ see \cite{LW}.
	It is   natural  to expect  similar rigidity result   regarding the Cheeger  constant  $h(M)$. 
	Building on   our previous work \cite{VV} relating the renormalized volume of convex co-compact hyperbolic 3-manifolds and the isoperimetric profile, we prove:
	\begin{theorem}\label{main_result}
		If $M$ be a convex co-compact hyperbolic 3-manifold which is either acylindrical or quasi-Fuchsian, then
		\[
		h(M) \leq  
		\frac2\alpha \approx 1.66711.
		\]
		where $\alpha$ is the unique positive solution of $\alpha=\coth\alpha$. Equality occurs if and only if, $M$ is $\Sigma\times \mathbb{R}$ endowed with a Fuchsian metric (\ref{fuchsian_metric}). 
	\end{theorem}

\begin{corollary}
If $\Omega$ is a  region of finite perimeter  in a Fuchsian 3-manifold $\Sigma\times \mathbb{R}$, then
\[
|\partial \Omega| \geq \frac{2}{\alpha}\, |\Omega|,
\]
with equality if, and only if, $\Omega$ is the slab bounded by  the slices $\Sigma\times \{\pm \alpha\}$.
\end{corollary}
 This       result   holds true in every  warped product $\Sigma\times_{\cosh}\mathbb{R}$, where $\Sigma$ is   closed.

\begin{corollary}\label{cor:isoperimetricFuchsian}
Let $\Omega_r$ denotes the slab $\Sigma\times[-r,r]$ in the Fuchsian 3-manifold. If $|\Omega|=|\Omega_r|$ and $r\geq \alpha$, then  \[|\partial\Omega|\geq |\partial \Omega_r|   \]  with equality if, and only if, $\Omega=\Omega_r$.
\end{corollary}

     \begin{remark}
    Related results for warped products $M\times_f I$.   Rafalski \cite{R} proved   in the setting of logarithmically convex warping functions  a relative isoperimetric inequality  for rooms bounded between  part  of a horizontal slice and  a geodesic graph over it. In particular, horizontal slices are isoperimetric when restrict to the class of entire geodesic graphs.  More generally,     Brendle \cite{B} proved Alexandrov's type results for  hypersurfaces with constant mean curvature:  it is shown that closed cmc surfaces  on the region $f^{\prime}>0$ are umbilic, provided that the warping factor $f$ satisfies certain structure conditions.
  \end{remark}

	\begin{remark}\label{rmk1}For any $\varepsilon>0$ there exist   quasi-Fuchsian $3$-manifolds  $M_{QF}$ such that \[\max\{h(M), \lambda_0(M)\}<\varepsilon.\]Indeed, when the convex core $C(M)$ has large volume, then $h(M)$ will become very small since $|\partial C(M)|<8\pi (g-1)$. Alternatively, it follows from \cite{S}  that $\lambda_0(M)=D(\Gamma)(2-D(\Gamma))$, where $D(\Gamma)$ is the Hausdorff dimension of the  limit set $\varLambda(\Gamma)$.  In contrast,  there exist  a  universal  lower bound for $\lambda_0$  in the class of  \textit{classical Schottky} 3-manifolds,  see \cite{D}.\end{remark}

	Below we list   similar rigidity results regarding the Cheeger constant.  The  statements    are straightforward consequences of  stronger  results in \cite{G,M,W2} and  the Cheeger inequality (\ref{cheeger_ineq}):
	\vspace{0.2cm}
	\begin{itemize}

		\item (\textit{Levy-Gromov \cite{G}}): If $M^{n}$ is a compact Riemannian manifold such that $Ric_M\geq n-1$, then
		\[
		h(M^n)\geq h(\mathbb{S}^n)
		\]
		with equality if, and only if, $M$ is isometric to the round sphere $\mathbb{S}^n$.
		\vspace{0.2cm}
		
		\item (\textit{X. Wang \cite{W2}}): Let $\widetilde{M}$ be the universal cover of a compact Riemannian manifold $M^n$. If  $Ric_M\geq -(n-1)$, then 
		\[
		h(\widetilde{M}) \leq n-1
		\] 
		with equality if, and only if, $\widetilde{M}$ is the hyperbolic space $\mathbb{H}^n$.
		\vspace{0.2cm}
		\item (\textit{O. Munteanu \cite{M}}): Let $\widetilde{M}$ be the universal cover of a compact  K\"{a}hler manifold $M^n$ of complex dimension $n$.  If $Ric_M\geq -2(n+1)$, then 
		\[
		h(\widetilde{M}) \leq 2n
		\] 
		with equality if, and only if, $\widetilde{M}$ is the  complex hyperbolic space $\mathbb{CH}^{n}$.
	\end{itemize}

	\section*{Acknowledgments}  
 
 We are thankful to Ian Agol for bringing this question to our attention. We are also thankful to the anonymous referee for their helpful feedback.

	\section{Preliminaries}

		\subsection{Cheeger isoperimetric constant} As mentioned in the introduction, we will be working with convex co-compact hyperbolic 3-manifolds $M$ which are either quasi-Fuchsian or acylindrical.

		\begin{lemma}\label{lemma2}
			If $M$ is convex co-compact hyperbolic 3-manifold with $\chi(\partial M)<0$, then any equidistant  foliation by compact convex sets  $U_r$  satisfies $|\partial U_r|< 2|U_r|$ and 
			\[
			\lim_{r\rightarrow \infty} \frac{|\partial U_r|}{|U_r|}=2.
			\]
		\end{lemma}
		\begin{proof}
			See Section 3 in \cite{VV}. 
		\end{proof}
		
		\noindent	 In particular, the Cheeger constant satisfies $0<h(M)<2$. In addition, we add that in the class of manifolds we are considering in this note, there exist a global minimizer  for $h(M)$, i.e., there is $\Omega$ such that $h(M) |\Omega|=|\partial \Omega|$  which  corresponds to a solution of the standard  isoperimetric problem.
		
		\begin{lemma}\label{lemma3}
			If $\Omega$ is an isoperimetric region of  mean curvature $H$ such that $|\partial \Omega|=h(M)|\Omega|$, then $2H=h(M)$ and $\partial \Omega$ is strongly stable.
		\end{lemma}
		
		\begin{proof}
			The Brane action  $\mathcal{F}= Area - h(M)\,Vol$ is non-negative and attains its global minimum  at $\Omega_V$. If $H$ is the mean curvature in the inward direction $-N$, then the first variation of $\mathcal{F}$ at $\Omega_V$ is given below
			\[
			\delta \mathcal{F}(fN)= \int_{\partial \Omega}f\,\big(2H -h(M) \big)\,d\Sigma=0 \quad\quad \forall 
			f\in C^{\infty}(\partial \Omega)
			\]
			Therefore, $h(M)=2H$. In particular, $H$ is positive. In addition,  the second variation of area and volume  implies
			\[
			0\leq \delta^2\mathcal{F}(f,f)= \int_{\partial \Omega} |\nabla f|^2 -\big(-2+|A|^2\big)f^2\,d\Sigma   \quad\quad \forall 
			f\in C^{\infty}(\partial \Omega).
			\]
			In other words, $\Sigma$ is a strongly stable constant mean curvature surface.
		\end{proof}
	
		\subsection{Cheeger constant of  warped  3-manifolds}
	In this section we study the Cheeger isoperimetric constant of $M=\Sigma\times \mathbb{R}$  endowed with the warped metric 
	\[
	g_0= dr^2 + \cosh^2(r) g_{\Sigma}.
	\]
	Here, $\Sigma$ is a general closed surface with a metric $g_\Sigma$ not necessarily hyperbolic.
	Let $\Sigma_r$ be the slice $\Sigma \times\{r\}$, then its area satisfies $A(r)=|\Sigma_0|\cosh^2(r)$ and $A^{\prime}(r)= 2|\Sigma| \cosh^2(r) \tanh(r)$. The first variation formula of area implies that the mean curvature of $\Sigma_r$, in the inward direction  $-\partial r$, is $H(r)=\tanh(r)$. Therefore,
	 \[\Delta r= 2\tanh(r).\]
In particular, 
	\[
\Delta \sinh(r)=\cosh(r)\Delta r + \sinh(r)|\nabla r|^2=3\sinh(r).
	\]
		\begin{proposition}\label{cheeger_fuchsian}
			The Cheeger constant of  the 3-manifold $(M,g_0)$ is equal to  
			\[
			h(M)=\inf_{x\in (0,+\infty)} \frac{\cosh^2(x)}{\int_{0}^{x}\cosh^2(t)dt}=\frac2\alpha \approx 1.66711.
			\]
		 It is uniquely attained by a slab  bounded by  the union of the slices $r=\alpha=\coth\alpha$  of the warped  metric $g_F$. 
		\end{proposition}

		\begin{proof}
		  We start by recalling that the bottom spectrum of $M$ is $\lambda_0=1$. The corresponding eigenfunction is $\varphi(p)= \cosh^{-1}(r)$. Indeed,
			\begin{eqnarray*}
			\Delta \cosh^{-1}(r)&=& -\frac{\sinh(r)}{\cosh^2(r)} \Delta r - \bigg(\frac{\cosh^3(r) - 2\sinh^2(r)\cosh(r)}{\cosh^4(r)}\bigg)|\nabla r|^2 \\
			  &=& -\frac{2\sinh^2(r)}{\cosh^3(r)} - \frac{1}{\cosh(r)} + \frac{2\sinh^2(r)}{\cosh^3(r)}=\,-\, \cosh^{-1}(r).
			\end{eqnarray*}
		Let us now compute $\Delta \tanh(r)$:
		\begin{eqnarray*}
		 \Delta \frac{\sinh(r)}{\cosh(r)}&=& \frac{1}{\cosh(r)} \Delta \sinh(r) + \sinh(r)\Delta \frac{1}{\cos(r)} +2 \langle \nabla \cosh^{-1}(r), \nabla \sinh(r)\rangle \\
		 &=& 3 \tanh(r) - \tanh(r) - 2\tanh(r)=0.
		\end{eqnarray*}
	Using that $\tanh(r)$ is harmonic we deduce that
	\begin{eqnarray}\label{main_eq}
	\Delta \big(r\tanh r\big) &=& \tanh(r)\Delta r+ r\Delta \tanh(r) + 2\langle \nabla r, \nabla \tanh(r)\rangle \nonumber \\
	&=& 2 \tanh^2(r) +2 \sech^2(r)=2
	\end{eqnarray}	
	Let $\Sigma=\partial \Omega$ be an isoperimetric surface  such that $|\partial \Omega|=h(M_F)|\Omega|$. Integrating (\ref{main_eq}) on $\Omega$ and applying the Divergence theorem:
	\begin{eqnarray*}
	2\,|\Omega|= \int_{\partial \Omega} \langle \nabla\big(r\tanh r\big), N\rangle \leq  \int_{\partial \Omega} \big(\tanh r +  r\sech^2 r\big) \langle \nabla r,N\rangle. 
	\end{eqnarray*}		
Note that $-1\leq \langle \nabla r, N\rangle \leq 1$ and  $f(r)=\tanh r + r\sech^2r \geq 0 $ if, and only if, $r\geq 0$.  The function $f(r)$ as a function of $r\in \mathbb{R}_+$  has a global maximum.  Indeed,
\begin{eqnarray*}
f^{\prime}(r)&=& \sech^2(r) + \sech^2(r) + 2r\sech r \big(-\sech^2 r\sinh r\big) \\
&=& \sech^2r\big(2 -2 r\tanh r\big).
\end{eqnarray*}
In other words, $f(r)$ has a unique point of maximum when $\tanh(\alpha)= \frac{1}{\alpha}$. Moreover, $f(r)\leq f(\alpha)= \alpha$. Applying this comparison in the sets $\{r> 0\}$ and  $\{r<0\}$ separately, we obtain 
\[
2\,|\Omega| \leq \alpha\, |\partial \Omega|.
\]
Moreover, the equality occur if, and only if $r(p)= \alpha$ for every $p\in \partial \Omega$. Therefore, $\partial\Omega=\Sigma\times \{-\alpha\}\cup \Sigma\times \{\alpha\}$.
	
		\end{proof}
		
		\begin{remark}\label{remark:allcheegerequal}
	We  apply Proposition \ref{cheeger_fuchsian} in the  spcial case of interest  which corresponds to  Fuchsian 3-manifolds: $\Sigma$ is a closed hyperbolic surface and the metric $g_F= dr^2+ \cosh^2 r\, g_{\Sigma}$ is  hyperbolic. In particular, $h(\Sigma\times_{\cosh} \mathbb{R})=2\alpha^{-1}$, regardless the topology  of the closed surface $\Sigma$. 
	 \end{remark}
 
 \begin{remark} It was proved in \cite{Wang2002,HMR} that if $M^n$ is conformally compact with conformal factor  $g_{\overline{M}}=f^2 g_M$ satisfying $f|_{\partial \overline{M}}=0$, $|\overline{\nabla} f|=1$ near $\partial \overline{M}$, and
 	\[
 	Ric_M \geq -(n-1) \quad \text{and}\quad S_M+ n(n+1)= o(f^2),
 	\]
 	then $h(M)=n$ as long as the Yamabe invariant  $\mathcal{Y}\big(\partial M,[\overline{g}]\big)\geq 0$. Let us compare this result in the warped metric  $\mathbb{S}^2 \times_{\cosh r} \mathbb{R}$. Recall that $h(\mathbb{S}^2\times_{\cosh r}\mathbb{R})<2$. The  Ricci and scalar curvature satisfy
	 \[
	 Ric(\partial r, \partial r)= -2, \quad  Ric(e_j,e_j)= -2\tanh^2 r, \quad \text{and}\quad R_M +6= 4 (1-\tanh^2 r).
	 \]
	In particular, $\mathbb{S}^2\times_{\cosh r} \mathbb{R}$ is locally asymptotic hyperbolic.   Note that $M$ is conformally equivalent to $\mathbb{S}^2\times [-\frac{\pi}{2},\frac{\pi}{2}]$ with conformal factor $f= 2\arctan\big(\tanh \frac{r}{2}\big)$. On the other hand, 
	 \[
	\lim_{f\rightarrow \frac{\pi}{2}^-} \frac{R_M(f)+6}{(f- \frac{\pi}{2})^2}= \lim_{f\rightarrow \frac{\pi}{2}^-} \frac{4(1-\tan^2 (f/2))}{(f - \frac{\pi}{2})^2}= +\infty.
	 \]
		\end{remark}
		
		\begin{lemma} \label{area_minimal}
			If   $\Sigma$ be  a  compact surface in a convex co-compact  3-manifold $M$, then
			\[
			\int_{\Sigma}\big(1-H^2\big)\,d\Sigma\leq 4\pi\big(g(\Sigma)-1\big)
			\]
			with equality if, and only if, the end of $M$ containing $\Sigma$ is Fuchsian.
		\end{lemma}
		\begin{proof}
			Since $\Sigma$ is a surface in a hyperbolic manifold, we have by the Gauss equation that $-1=K + \frac{|\mathring{A}|^2-2H^2}{2}$.  The statement then follows from the Gauss-Bonnet Theorem and $|\mathring{A}|^2\geq 0$. If equality occur, then $H<1$,  $\mathring{A}=0$ and $\Sigma$ is hyperbolic of curvature $-1+H^2$. The equidistant flow starting at $\Sigma$ will produce a  foliation of $M$ by totally umbilical surfaces. The metric $g_M$ in this coordinates is $g_M= dr^2+ \cosh^2(r)g_{\Sigma}$, see \cite{U}.
		\end{proof}
	
	\begin{remark}
	In case $\Sigma$ is strongly stable constant mean curvature surface, then we have  the strict lower bound estimate:
	\begin{eqnarray}\label{lower_bound_energy}
	\int_{\Sigma} \big(1-H^2\big) d\Sigma > 2\pi \big(g(\Sigma)-1\big).
	\end{eqnarray}
	Indeed, the strongly stability applied to the test function $f=1$ implies that
	$
	\int_{\Sigma} \big(2- |A|^2\big)d\Sigma \geq 0.
	$
Now we apply the Gauss-Bonnet  Theorem to the Gauss equation $|A|^2=4H^2-2K-2$.
	 The equality in (\ref{lower_bound_energy}) is only attained when $H=1$ and $\Sigma$ is a horotorus (quotient of an horosphere). This follows since equality would imply that $f=1$ is an eigenfunction of the Jacobi operator $L=\Delta -2+ |A|^2$. In particular, $|A|^2=2$. As $\Sigma$ contains an umbilical point, $\Sigma$ must be totally umbilical with $H=1$. Hence, $\Sigma$ is a horotorus.
	\end{remark}
		
		\section{Proof of the main result}
		
		\begin{definition} For a convex co-compact hyperbolic $3$-manifold $M$ which is either quasi-Fuchsian or acylindrical, the outermost region  $\Omega_0$ is defined as the largest compact region in $M$ such that $\partial \Omega_0$ is a minimal surface and $M-\Omega_0$ contains no other minimal surface homologous to the conformal boundary $\partial M$.
		\end{definition}
		
		The outermost isoperimetric profile $I_M: (0,+\infty)\rightarrow \mathbb{R}$ is defined as follows:
		\begin{eqnarray*}
			I_M(V)=\inf \{|\partial \Omega|\,:\, \Omega_0\subset \Omega \quad \text{and}\quad vol(\Omega-\Omega_0)=V\}
		\end{eqnarray*}
		It is well known that $I_M$ is an strictly increasing function, see for instance \cite{VV}. The standard isoperimetric profile $J_M$ of $M$ is defined as follows
		\begin{eqnarray*}
			J_M(V)=\inf \{|\partial \Omega|\,:\, \Omega\subset M \quad \text{and}\quad vol(\Omega)=V\}
		\end{eqnarray*}
	We also recall that  $I_{TG}$ denotes  the outermost isoperimetric profile of the model hyperbolic metric with totally geodesic convex core $\Omega_{TG}$.
	
	An important geometric invariant in the context of convex co-compact hyperbolic 3-manifold is the Renormalized Volume, denoted by $V_R (M)$. By the duality between conformal metrics at infinity		($\partial_{\infty}\mathbb{H}^3$) and embedded surfaces in $\mathbb{H}^3$, one can construct a region $N\subset M$ so that $V_R(M)$ is equal to the volume of $N$ minus half of the integral of the mean curvature of $\partial N$, see \cite{BBB,VP}. In  \cite{VV}, we proved that  $V_R(M)$ is determined by the  isoperimetric structure at infinity of $M$ via $J_M$. Moreover,
		\begin{eqnarray}\label{ren_vol}
		V_R(M) - |\Omega_0| = \frac{1}{2} \lim_{V\rightarrow \infty} \bigg(I_{TG}(V)- I_M(V)\bigg).
		\end{eqnarray}
	Using this result and the   assumption $V_R(M)> |\Omega_0|$, we   proved   $I_M(V)\leq I_{TG}(V)$ for every $V$. This assumption holds true for example at an open set of metrics that include the almost Fuchsian \cite{U}.

The variational properties of  $V_R(M)$ in the space of convex co-compact hyperbolic 3-manifolds have been  studied in the literature. In particular, the functional $V_R(M,g)$ attains its  minimum  at the Fuchsian locus, see \cite{BBB,VP}. Building on  this result,  we proved regardless the sign of   $V_R(M)-|\Omega_0|$ the following comparison:		
		\begin{theorem}[Vargas Pallete-Viana \cite{VV}]\label{TheoremVV}
			Let $M$ be a convex co-compact hyperbolic 3-manifold which  is either  quasi-Fuchsian  or acylindrical and let  $\Omega_0$ be  the  outermost region. Then
			\[
			I_M(V)\leq I_{TG}(V+|\Omega_0|-|\Omega_{TG}|),
			\]
for every volume $V>0$.			If equality occurs for some $V$, then  $\Omega_0=\Omega_{TG}$.
		\end{theorem}
		By the special case in the main theorem in \cite{AST} we have that  $|\Omega_0|\geq |\Omega_{TG}|$. 
		\vspace{0.2cm}
		
We  show that the outermost isoperimetric  profile comparison above   is enough to prove Theorem \ref{main_result}:

		\begin{proof}[Proof Theorem \ref{main_result}]Let $\Omega_{g,h}$ the Cheeger region of the Fuchsian metric $g_F$ on $\Sigma_g\times\mathbb{R}$:
			\[
			|\partial \Omega_{g,h}|=h_F\, |\Omega_{g,h}|.
			\]
			By Lemma \ref{cheeger_fuchsian},  $\Omega_{g,h}$ is bounded by two  slices of the warped metric $g_F$. Let $\lbrace \Sigma_{g_i} \rbrace$ denote components of $\partial M$ with their respective genus as a sub-index, and let $\Omega_{g_i,h}$ be the respective Cheeger regions in $\Sigma_{g_i}\times\mathbb{R}$.
			As noted, we have  by the special case of the main theorem in \cite{AST} that  $|\Omega_0|> |\Omega_{TG}|$ unless $\Omega_0=\Omega_{TG}$.

			If $|\Omega_0|-|\Omega_{TG}|\geq \frac12\sum_i |\Omega_{g_i,h}|$, then
				\[
				h(M) \leq \frac{|\partial \Omega_0|}{|\Omega_0|} < \frac{\sum_i \frac12|\partial\Omega_{g_i,h}|}{|\Omega_0|} \leq \frac{\sum_i \frac12|\partial\Omega_{g_i,h}|}{|\Omega_0|-|\Omega_{TG}|} \leq \frac{\sum_i |\partial\Omega_{g_i,h}|}{\sum_i |\Omega_{g_i,h}|}=h_F.
				\]
			The first inequality follows from the definition of $h(M)$, second inequality follows from Lemma \ref{area_minimal} and the fact that $\partial\Omega_{g,h}$ has two components, the third inequality from $|\Omega_{TG}|\geq 0$, the fourth inequality follows by the assumption $|\Omega_0|-|\Omega_{TG}|\geq \frac12\sum_i |\Omega_{g_i,h}|$, and the final equality follows by Remark \ref{remark:allcheegerequal}. Hence, $h(M)\leq h_F$ and equality does not occur.

		If $|\Omega_0|-|\Omega_{TG}|< \frac12\sum_i |\Omega_{g_i,h}|$, then   $\frac12\sum_i |\Omega_{g_i,h}|=|\Omega_0|-|\Omega_{TG}|+ V$ for some  $V>0$. We will bound $h(M)$ by applying Theorem \ref{TheoremVV}. For this we need to deal with the model metrics with Fuchsian ends and then apply Lemma \ref{cheeger_fuchsian}. Namely, we have:
				\begin{eqnarray*}
					h(M)&\leq& \frac{J_M(V+|\Omega_0|)}{V+|\Omega_0|} \leq	\frac{I_M(V)}{V+ |\Omega_0|}   \\
					\\
					&\leq& \frac{I_{TG}(|\Omega_0|-|\Omega_{TG}|+V)}{V+|\Omega_0|} 
					\leq \frac{I_{TG}(|\Omega_0|-|\Omega_{TG}|+V)}{V+|\Omega_0|-|\Omega_{TG}|} \\
					\\
					&=& \frac{I_{TG}(\frac12\sum_i |\Omega_{g_i,h}|)}{\frac12\sum_i |\Omega_{g_i,h}|} 
					=  \frac{\sum_i \frac12 I_F(|\Omega_{g_i,h}|)}{\frac12\sum_i |\Omega_{g_i,h}|} =\frac{\sum_i h_F\,|\Omega_{g_i,h}|}{\sum_i |\Omega_{g_i,h}|}= h_F.
				\end{eqnarray*}
				The first two inequalities follow from the definitions of $h(M)$, $J_M(\cdot)$ and $I_M(\cdot)$. The third inequality follows from Theorem \ref{TheoremVV}, while the fourth inequality follows from $|\Omega_{TG}|\geq 0$. The first equality follows from the definition of $V$,  the second equality follows from the fact that each end of  the totally geodesic model  is Fuchsian and the third equality follows from Lemma \ref{cheeger_fuchsian}. The equality $h(M)=h_F$ implies  $|\Omega_{TG}|=0$ and also that   $\Omega_0=\Omega_{TG}$ by Theorem \ref{TheoremVV}. 
		\end{proof}

  \begin{proof}[Proof of Corollary \ref{cor:isoperimetricFuchsian}]  
  There exists $\varepsilon>0$ such that if  $\Omega$ is an isoperimetric region of volume $|vol(\Omega)-vol(\Omega_{\alpha})|<\varepsilon$, then $\Omega$ is a graphical over $\Omega_{\alpha}$ by the uniqueness of $\Omega_{\alpha}$. Since $H_r$ is increasing, we have by the Maximum Principle that  $\Omega=\Omega_r$. On the other hand, if $vol(\Omega)$ is sufficiently large, then $\Omega=\Omega_r$, see \cite{VV}.  
Let $\beta(V)$ denote the area-volume profile with respect to the foliation $\Omega_r$. Hence, $\beta(V)=J(V)$  for volumes near $vol(\Omega_{\alpha})$ and for sufficiently large volumes. Hence, there exists   a  point of   global maximum  $V_m$ for the function \[\varphi(V)=(\beta-J)(V)\] on  $(vol(\Omega_{\alpha}),\infty)$. We denote by $\Omega_m$ an isoperimetric region of volume $V_m$, and let $r_0>\alpha$ be so that $\beta(V_m) = |\partial \Omega_{r_0}|$. Let $F_t$ the  equidistant variation of  $\partial \Omega_m$ and $f(V)$ the corresponding area-volume profile with respect to this family. It follows that $V_m$ is  a local  point of maximum for the function $\overline{\varphi}=\beta-f$. Hence, \[\overline{\varphi}^{\prime}(V_m)=\beta^{\prime}(V_m)-f^{\prime}(V_m)=0.\]As $f$ and $\beta$ are smooth functions, we obtain\[2 H_{r_0}=\beta^{\prime}(V_m)=f^{\prime}(V_m)=2 H_m. \]
Note that $|\Omega_{r_0}|=V_m=|\Omega_m|$. In particular, there exists $p\in \partial \Omega_m$ such that  $r\geq r_0$. If $r_m$ is the  maximum height over $\partial \Omega_m$, then the Maximum Principle implies that $H_m\geq H_{r_m}>H_{r_0}$, unless $\Omega_m=\Omega_{r_0}$. Therefore, $J(V)=\beta(V)$ for every $V\geq vol(\Omega_{\alpha})$.
\end{proof}

\end{document}